\documentclass[10pt,a4paper]{amsart}
\usepackage{graphicx}        
\usepackage{multicol}        
\usepackage{amsmath,amssymb,amsfonts}        

\usepackage[all]{xy}

\usepackage[mathscr]{euscript}

\usepackage{fullpage}

\def\bysame{\leavevmode\hbox to3em{\hrulefill}\thinspace}

\newcommand{\beq}{\begin{equation}}
\newcommand{\eeq}{\end{equation}}

\newcommand{\p}{\mathsf{p}}
\newcommand{\m}{\mathsf{m}}

\newcommand{\A}{\mathcal{A}}

\newcommand{\T}{\mathbb{T}}

\newcommand{\K}{\mathcal{K}}
\newcommand{\CP}{\mathbb{CP}}
\newcommand{\WP}{\mathbb{WP}}

\newcommand{\cO}{\mathscr{O}}

\newcommand{\Z}{\mathbb{Z}}
\newcommand{\C}{\mathbb{C}}

\newcommand{\N}{\mathbb{N}}

\newcommand{\Cs}{C$^*$-}

\renewcommand{\hat}{\widehat}

\DeclareMathOperator{\Mat}{Mat}

\DeclareMathOperator{\End}{End}

\newtheorem{thm}{Theorem}[section]

\newtheorem{prop}[thm]{Proposition}

\newtheorem{exmp}{Example}[section]

\title{Gysin exact sequences for quantum weighted lens spaces}
\author{Francesca Arici}
\address{Institute for Mathematics, Astrophysics and Particle Physics, Faculty of Science, Radboud University Nijmegen, Heyendaalseweg 135, 6525AJ Nijmegen, The Netherlands}
\email{f.arici@math.ru.nl}
\thanks{This research was partially supported by NWO under the VIDI-grant \mbox{016.133.326}. }
\begin{document}


\begin{abstract}
We describe quantum weighted lens spaces as total spaces of quantum principal circle bundles, using a Cuntz-Pimsner model. The corresponding Pimsner exact sequence is interpreted as a noncommutative analogue of the Gysin exact sequence. We use the sequence to compute the K-theory and K-homology groups of quantum weighted lens spaces, extending previous results and computations due to the author and collaborators. \end{abstract}

\maketitle

\tableofcontents
\parskip 1ex
\linespread{1.1} 

\maketitle

\section{Introduction}

Quantum lens spaces, both weighted and unweighted, have been the subject of increasing interest in the last years. They are Cuntz-Krieger algebras of a directed graph \cite{HS03} and have played an important role in the classification program of \mbox{\Cs algebras} \cite{ERRS16}. Using graph algebra techniques their K-theory groups have been computed recently in \cite{BS16} under very general assumptions on the weight. From a more geometric point of view, they have a natural structure of noncommutative principal circle bundles over quantum weighted projective spaces \cite{ABL15,AKL16,BF12,BF15,SV15} and can thus be interpreted as a deformation of their classical counterparts. 
In this paper we focus on the noncommutative topology of quantum weighted lens spaces, realising them as Cuntz-Pimsner algebras of self-Morita equivalence bimodules. This allows us to compute their K-theory and K-homology groups, using different techniques than those in \cite{BS16}.

Being graph algebras, quantum weighted lens spaces admit a Cuntz-Pimsner model where the coefficient algebra is the algebra of functions on the vertex space. This picture is very well suited to encode the dynamical information contained in the graph, but has the disadvantage that the fixed point algebra for the natural gauge action does not agree with the coefficient algebra. In the Cuntz-Pimsner model we employ here, which comes from the geometric analogy described above, the coefficient algebra is the algebra of functions on a quantum weighted projective space and the resulting Cuntz-Pimsner algebra can be thought of as the total space of a noncommutative circle bundle. 

The associated six term exact sequences can then be interpreted as the operator theoretic counterpart of the Gysin exact sequence for circle bundles. Under some mild assumptions on the weight, we will describe the K-theory and K-homology of quantum weighted lens spaces of any dimension, thus extending the results of \cite{ABL15} and \cite{AKL16}.

\subsection*{Acknowledgements}We thank the mathematical research institute MATRIX in Australia and the organisers of the workshop "Refining \Cs algebraic invariants using KK-theory", where part of this research was performed. This work was motivated by discussions with Efren Ruiz about the structure of graph algebras. We thank Adam Rennie for helpful discussion and for his hospitality at the University of Wollongong, where part of this work was carried out. Finally, the author would like to thank Francesco D'Andrea, Giovanni Landi, Bram Mesland and Walter van Suijlekom for helpful comments on an early version of this work. 

\section{Quantum weighted projective and lens spaces}
In this section we describe the coordinate algebras of weighted projective and lens spaces, as described in \cite{BF15, DAL15, ADL16} and their \Cs completions, which were extensively studied in \cite{BS16}.

Classically, weighted projective and lens spaces are quotients of odd-dimensional spheres by actions of the circle and of a finite cyclic group, respectively. The same is true upon replacing the sphere by a \emph{quantum} sphere.

We recall from \cite{VS91} that the coordinate algebra of the quantum odd-dimensional sphere 
$\cO ( S^{2n+1}_q)$ is the universal $*$-algebra with generators the 
$n+1$ elements $\{z_i\}_{i=0,\ldots,n}$ and relations:
\begin{align*} 
z_iz_j &=q^{-1}z_jz_i && 0\leq i<j\leq n \;,  \\
z_i^*z_j &=qz_jz_i^*  &&  i\neq j \;,  \\
[z_n^*,z_n] =0, \qquad [z_i^*,z_i] &=(1-q^2)\sum_{j=i+1}^n z_jz_j^* && i=0,\ldots,n-1 \;,   
\end{align*}
and a sphere relation:
$$
z_0z_0^*+z_1z_1^* +\ldots+z_nz_n^*=1 \;.
$$
The notation of \cite{VS91} is obtained by setting $q=e^{h/2}$. 

A \emph{weight vector} $\m=(m_0, \dots, m_n)$ is a finite sequence of positive integers, called \emph{weights}. A weight vector is said to be \emph{coprime} if $\gcd(m_0,\dots,m_n)=1$; and it is \emph{pairwise coprime} if $\gcd(m_i,m_j)=1$, for all $i\neq j$.

For any weight vector $\m=(m_0,\dots, m_n)$,
we define a weighted circle action $\lbrace \sigma^{\m}_\xi \rbrace_{\xi \in \T^1}$ on the quantum sphere, given on generators by
\begin{equation}
\label{u1act}
\sigma^{\m}_\xi (z_i) =  \xi^{m_i} z_1 \qquad \xi \in \T^1,
\end{equation}
The grading induced by this action is equivalent to that obtained by declaring each $z_i$ to be of degree $m_i$ and $z_i^*$ of degree $-m_i$.

The degree zero part or, equivalently, the fixed point algebra for the action, is the coordinate algebra of the \emph{quantum $n$-dimensional weighted projective space} associated with the weight vector $\m$, and denoted by $\cO(\WP_q^n(\m))$. 

By \cite[Lemma 3.2]{DAL15} a set of generators for the algebra $\cO(\WP_q^n(\m))$ is given by the elements
\[ \begin{split}z_iz_i^*, \quad \forall i=0, \dots, n \\ 
z^{\mathbf{k}} := z_0^{k_0} \cdots z_n^{k_n}, \quad \forall \ \mathbf{k} \in \Z^n : \mathbf{k}\cdot \m  := k_0m_0 + \cdots + k_n m_n = 0.
\end{split}\]
Note that such a set of generators is in general not minimal. 

For some particular classes, one gets a complete characterisation for the generators of the algebra $\cO(\WP^n_q(\m))$. Indeed, we have two classes of weighted projective spaces, in some sense orthogonal to each other, for which it is possible to describe the generators and the representation theory.

The first class consists of those weighted projective spaces for which the weight $\m$ is of the form $\m= \p^{\sharp}$ for $\p$ pairwise coprime, where $\p^{\sharp}$ is defined as the weight vector whose $i$-th component is equal to $\prod_{j\neq i} p_j$. Classically those are the weighted projective spaces that are isomorphic, as projective varieties, to the unweighted projective space $\CP^n$.
By \cite[Theorem 3.8]{DAL15}, having such a weight is a necessary and sufficient condition for the algebra $\cO(\WP^n_q(\m))$ to be generated by the elements
\[a_{i,j}:= z_i^{*m_{j:i}} z_j^{m_{i:j}}, \quad m_{i:j} = m_i / \gcd (m_i,m_j) \quad \forall i, j =0, \dots, n.\]

The second class of examples, that goes in another direction with respect to the class we just described, is that of the multidimensional teardrops \cite{BF15}, that are obtained for the weight vector $m=(1,\dots, 1, l)$ having all but the last entry equal to $1$. As described in \cite[Lemma~6.1]{BF15}
the algebra $\cO(\WP^n_q(1, \dots, m))$ is generated, as a *-algebra, by the elements 
\[
b_{i,j}:= (z_i^*) z_j \quad \mbox{and} \quad 
{c}_{\mathbf{l}}:= z_0^{l_0} \cdots z_{n-1}^{l_{n-1}} z_n^*  ,
\]
for $0 \leq i \leq j \leq n-1  $ and $\mathbf{l} \in \N^n$ such that $ \sum_{i=0}^{n-1} l_i= m$.

 As a particular case of both constructions, for $\m=(1,\ldots,1)$ one gets 
the coordinate algebra $\cO(\mathbb{CP}^n_q)$ of the quantum projective space $\CP^n_q$. This is the  $*$-subalgebra of $\cO(S^{2n+1}_q)$ generated by the elements $p_{ij}:=z_i^*z_j$ for $i,j=0,1,\ldots,n$.

Let now $N$ be a fixed positive integer. By restriction $\cO(S^{2n+1}_q)$ admits an action of the cyclic group $\Z_N$ given by
\[
\sigma^{1/N}_{\m}(\zeta) z_i = \zeta^{m_i} z_i,
\]
where $\zeta$ is a generator of $\Z_N$.

The coordinate algebra of the quantum lens space $\cO(L^{2n+1}_q(N;\m))$ is defined as the fixed point algebra for this action:
\begin{equation}
\cO(L^{2n+1}_q(N; \m)) := \cO(S^{2n+1}_q)^{\Z_N}. 
\end{equation}

\subsection{Principal bundle structures}

In noncommutative geometry the notion of a free action of a quantum group $H$ on $\A$ is translated into the notion of having a principal coaction on $\A$, which in algebraic terms amounts to having a Hopf-Galois extension. As described in \cite[Theorem~8.1.7]{M93}, having a strongly graded algebra over a group G is equivalent to having a Hopf-Galois extension over the group algebra $\C G$. In this work we will only focus on classical Abelian groups; in that case principality is equivalent to the induced grading over the Pontryagin dual being \emph{strong} .

For a group $G$, a $G$-graded algebra $\A$ is an algebra that decomposes as a direct sum $\A = \oplus_{g \in G} \A_g$, with $\A_g \A_h \subseteq \A_{gh}$. Whenever $\A_g \A_h = \A_{gh}$ for all $g, h \in G$, one says that the grading is strong. Note that by it is enough to check this condition on a set of generators of the group. 

By \cite[Lemma~2.1]{BF15} strong gradings are preserved under extensions of Abelian groups, i.e. given an exact sequence \[ \xymatrix{0 \ar[r] &  K \ar[r]^-{\varphi} & G \ar[r]^-{\pi}& H \ar[r] & 0}, \]
a $G$-graded algebra $\A$ is strongly graded if the induced $H$-grading on $\A$ and the induced $K$-grading on $\A^K := \bigoplus_{k \in K}\A_{\varphi(k)}$
are strong.

In our case we will be dealing with the group $\Z =\hat{\T}$ and the finite group cyclic group $\Z_N = \hat{\Z_N}$, so we will be interested in the short exact sequence
 \[ \xymatrix{0 \ar[r] &  \Z \ar[r]^-{N\cdot} & \Z \ar[r]^-{\pi}& \Z_N \ar[r] & 0.}\]

The $\Z_N$ action on $\cO(S^{2n+1}_q)$ induces a $\Z_N$ grading which is strong, and $\cO(L^{2n+1}_q(N;\m))$ is the degree-zero subalgebra with respect to that grading. The lens space $\cO(L^{2n+1}_q(N;\m))$ is a $\Z$-graded algebra with respect to the grading induced by that of $\cO(S^{2n+1}_q)$, by saying that $x \in \cO(L^{2n+1}_q(N;\m))$ has degree $n$ if and only if it has degree $nN$ in $\cO(S^{2n+1}_q)$. The degree-zero part is given by the coordinate algebra of the quantum projective space $\cO(\WP^n_q(\m))$. The induced grading is not always strong. However, by \cite[Proposition~4.2]{BF15}, for $N_{\m} := \prod_{i=1}^n m_i$, the algebra $\cO(L^{2n+1}_q(N_{\m};\m))$ is strongly $\Z$-graded.
As a consequence, the coordinate algebra of the quantum weighted lens space $  \cO(L^{2n+1}_{q} (N_{\m}; \m))$ has the structure of a quantum principal circle bundle over the $n$-dimensional weighted projective space $\cO (\WP^n_q(\m))$.

Furthermore, by \cite[Proposition~4.6]{AKL16} the coordinate algebra of the quantum weighted lens space $  \cO(L^{2n+1}_{q} (d\cdot N_{\m}; \m))$ also has the structure of a quantum principal circle bundle over the $n$-dimensional weighted projective space $\cO (\WP^n_q(\m))$. This can be seen as a consequence of the aforementioned Lemma~2.1 of \cite{BF15}.

\subsection{\Cs completions}

The \Cs algebra of the odd-dimensional quantum sphere $C( S^{2n+1}_q)$ is the completion of the *-algebra $\cO(S^{2n+1}_q)$ in the universal \Cs norm. This \Cs algebra can be realised as a graph \Cs algebra. 

The \Cs algebra of the quantum weighted projective space is defined as the fixed point algebra for the circle action on $C( S^{2n+1}_q)$ obtained by extending $\sigma$. A complete characterisation of those \Cs algebras is not available at the moment; partial results were obtained in \cite{BS16} for a large class of weighted lens spaces, those with weight vector $\m$ satisfying $\gcd (m_j, m_n) =1$ for at least one $j<n$. By \cite[Proposition~3.2]{BS16} there exists an exact sequence of \Cs algebras
\beq
\label{eq:extW}
 \xymatrix{0 \ar[r] & \K^{\oplus m_n} \ar[r] & C(\WP^n(\m)) \ar[r] & C(\WP^{n-1}_q(\m_n)) \ar[r] & 0,}
  \eeq
 where $\m_n$ denotes the weight vector $(m_0, \dots, m_{n-1})$.
 
The K-theory groups of the \Cs algebraic weighted projective spaces can be computed by iterative use of the extension \eqref{eq:extW} under suitable assumptions on the weight vector $\m$.
 \begin{prop}[{\cite[Corollary 3.2]{BS16}}]
 \label{prop:Kthweigh}
 \label{ass:weight}
Let $\m$ be a weight vector with the property that for each $j \geq 1$ there exists $i<j$ such that $\gcd (m_i, m_j) =1$. Then the K-theory groups of the quantum weighted projective spaces are given by
\[K_0(C(\WP^n_q(\m))= \Z^{1+\sum_{i=1}^{n} m_i}, \quad K_1(C(\WP^n_q(\m))=0. \]
\end{prop}
The \Cs algebraic quantum lens space is defined as the fixed point algebra by the action of $\Z_N$. By constructing a conditional expectation for the $\Z_N$-action, one can show that it agrees with the closure of the algebraic quantum lens space $\cO(L^{2n+1}_q(N; \m))$ with respect to the universal \Cs norm on $C(S^{2n+1}_q)$. It is isomorphic to the \Cs algebra of a directed graph.

\section{A Cuntz-Pimsner model for quantum lens spaces}
Cuntz-Pimsner algebras \cite{Pi97} are universal \Cs algebras constructed out of a \Cs correspondence $E$ over a \Cs algebra $B$. They encompass a a large class of examples, like crossed product by the integers, Cuntz and Cuntz-Krieger algebras \cite{CuKr:TMC,Cun:SGI}, graph algebras and C*-algebras associated to a partial automorphism \cite{Exe:CPP}. We now give a simple description of this algebra for the case of interest for this work. 

Under the assumptions that $B$ is unital and that $E$ is a self-Morita equivalence bimodule, i.e. we have left action implemented by an isomorphism $\phi: B \to \End_B(E)$, the Cuntz-Pimsner algebra $O_E$ admits a description in terms of generators and commutation relations. This construction, which can be found for instance in \cite[Section 2]{KPW}, works for any finitely generated projective module over a unital \Cs algebra and relies on the existence of a finite frame for the module $E$, i.e. a finite set of elements $\lbrace \xi_i \rbrace_{i=1}^{n}$ of $E$ satisfying 
$$
\xi = \sum_{j=1}^{n}  \eta_j \langle \eta_j, \xi \rangle_{B}.
$$
for any $\xi \in E$. 

The algebra $O_E$ is realised as the universal \mbox{\Cs algebra} generated by $B$ together with 
$n$ operators $S_1, \dots, S_n$, satisfying
\begin{align}
& S_i^* S_j = \langle{\eta_i, \eta_j}\rangle_{B}, \quad \sum\nolimits_j S_j S_j ^* = 1, \quad \mbox{and} \quad 
b S_j = \sum\nolimits_{i}S_i \langle{\eta_i, \phi(b) \eta_j}\rangle_{B},
\end{align}
for $b \in B$, and $j=1,\dots, n$. 

\begin{exmp}
The module of sections of the tautological line bundle $\mathcal{E}$ over the quantum projective line is a self-Morita equivalence bimodule over the algebra $C(\CP^1)$. The corresponding Cuntz-Pimsner algebra $O_{\Gamma(\mathcal{E})}$ is isomorphic to the algebra of continuous functions on the three sphere $C(S^3)$. 
\end{exmp}

More generally, the Cuntz-Pimsner algebra $O_E$ of a self-Morita equivalence bimodule can be thought of as the algebra of continuous functions on the total space of a quantum principal circle bundle. While the commutative version of this analogy was spelled out in \cite{GG13}, the more general case of quantum principal circle bundles was described in \cite{AKL16}. We also refer to the review article \cite{ADL16} for more details and recall the salient points here.

Given a \Cs algebra $A$ together with a strongly continuous circle action $\sigma := \lbrace \sigma_{\xi} \rbrace_{\xi \in \T^1}$, we define the $n$-th spectral subspace as 
\[A_{(n)} := \lbrace a \in A \ \vert \ \sigma_{\xi} (a) = \xi^{-n} a \quad \forall \xi \in \T^1 \rbrace.\]

Then the invariant subspace $A_{(0)} \subseteq A$ is a \Cs subalgebra and each $A_{(n)}$ is a Hilbert \Cs bimodule over $A_{(0)}$. If the module $A_{(1)}$ is a self-Morita equivalence bimodule, which is equivalent to the $\Z$ grading given by the spectral subspaces being strong, then the action $\sigma$ is said to be saturated. Then by \cite[Prop. 3.5]{AKL16} the Cuntz-Pimsner algebra $O_{A_{(1)}}$ of the first spectral subspace $A_{(1)}$ is isomorphic to the algebra $A$.

Let see what this means in our examples: if we denote by
$E$ the first spectral subspace of $C(S^{2n+1}_q)$ for the weighted action of $\T^1$, then we get that the $C^*$-algebra
$C(L^{2n+1}_{q} (N_{\m}; \m))$ is isomorphic to the Pimsner algebra $O_E$ over $C(\WP^n_q(\m))$ associated to $E$.

More generally, by \cite[Thm. 3.9]{AKL16}. for any $d\geq 1$, the \Cs algebraic lens space
$C(L^{2n+1}_{q} (d\cdot N_{\m}; \m))$ is isomorphic to the Pimsner algebra $O_E$ associated to the module $E^{\otimes d}$ over $C(\WP^n_q(\m))$ .

As particular cases, $C(S^{2n+1}_q)$ is a Pimsner algebra over $C(\CP^n_q)$, and more generally the \emph{unweighted} lens space $C(L^{2n+1}_q(d; \underline{1}))$, for the weight vector with entries identically one, is a Pimsner algebra over $C(\CP^n_q)$ for any $d\geq 1$. Those algebras and their K-theory group were the subject of \cite{ABL15}.

\subsection{Six-term exact sequences}
For a Pismner algebra one has natural exact sequences in bivariant K-theory, relating the KK-groups of the Pimsner algebra $O_E$ with those of the base space algebra $B$. Those sequences were constructed by Pimsner, see \cite[Theorem~4.8]{Pi97} and arise as six-term exact sequences associated to a semisplit extension of \Cs algebras in which the Pimsner algebra is the quotient, the ideal is Morita equivalent to the base and the middle algebra is KK-equivalent to the base. Using those identifications, the resulting exact sequences in bivariant K-theory read:

\[
 \xymatrix{
KK_0(C, B) \ar[r]^{\otimes_B(1 - [E])} & KK_0(C,B) \ar[r]^-{i_*}& KK_0(C,O_E) \ar[d]^{\partial} \\
KK_1(C,O_E)\ar[u]^{\partial} & KK_1(C,B) \ar[l]^-{i_*} & KK_1(C,B) \ar[l]^{\otimes_B(1 - [E])}
 }
\]
 and 
 \[
\xymatrix{
 KK_0(B,C)\ar[d]_{\partial} & KK_0(B,C) \ar[l]_{(1 - [E])\otimes_B} & KK_0(O_E,C) \ar[l]_-{i^*} \\
KK_1(O_E,C) \ar[r]_{i^*}&  KK_1(B,C) \ar[r]_{(1 - [E])\otimes_B} & KK_1(B,C) \ar[u]_{\partial}}.
\]

A crucial role is played by the Kasparov product with the class of the identity $1 \in KK(B,B)$ minus the class of the bimodule $[E] \in KK(B,B)$. 

The connecting homomorphism $\partial$ is implemented by taking the Kasparov product with the class of the defining extension. An unbounded representative for this extension class was constructed in \cite{RRS15} and later generalised in \cite{GMR15}. A treatment of the non-unital case is in \cite{AR16}. 

In the case of self-Morita equivalence bimodules these could be considered as a generalization of the classical \emph{Gysin sequence} in $K$-theory (cf.\cite[IV.1.13]{Ka78}) 
for the \emph{noncommutative} line bundle $E$ over $B$, with the Kasparov product with $1 - [E]$ playing the role of the cup product with 
\emph{Euler class} $\chi(E):=1 - [E]$ of the line bundle $E$.

Examples of Gysin sequences in $K$-theory were given in \cite{ABL15} for line bundles over quantum projective spaces leading to a class of quantum lens spaces. 
These examples were generalized later in \cite{AKL16} for a class of quantum lens spaces as circle bundles over quantum weighted projective spaces with arbitrary weights.

To ease our notation, we let 
$C(L_q(d)) := C(L^{2n+1}(d\cdot N_{\m},\m))$. Also, $E$ will denote the Hilbert \Cs module given by the first spectral subspace for the weighted circle action on $C(S^{2n+1}_q)$.

Then, given any separable $C^*$-algebra $C$, we obtain two six term exact sequences in KK-theory.
\begin{equation}\label{eq:6tesKth}
\xymatrix{
KK_0(C, C(\WP_q^{n}(\m))) \ar[r]^-{(1 - [E^{\otimes d}])\otimes} & KK_0(C,C(\WP_q^{n}(\m))) \ar[r]^{i_*} & KK_0\big(C,C(L_q(d))\big) \ar[d]^{\partial} \\
KK_1(C,C(L_q(d))) \ar[u]^{\partial} & KK_1(C,C(\WP_q^{n}(\m))) \ar[l]^{i_*} & KK_1(C,C(\WP_q^{n}(\m)) \ar[l]^-{(1 - [E^{\otimes d}])\otimes} }
\end{equation}
and
\begin{equation}\label{eq:6tesKhom}
\xymatrix{
KK_0(C(\WP_q^{n}(\m)),C)\ar[d]_{\partial} & KK_0(C(\WP_q^{n}(\m)),C) \ar[l]_{\otimes(1 - [E^{\otimes d}])} & KK_0\big(C(L_q(d)), C\big) \ar[l]_{i^*} \\
KK_1\big(C(L_q(d)),C\big) \ar[r]_{i^*}& KK_1(C(\WP_q^{n}(\m)),C) \ar[r]_{\otimes(1 - [E^{\otimes d}])}& KK_1(C(\WP_q^n(\m)),C) \ar[u]_{\partial} }
 \end{equation}
where $i_*$ and $i^*$ are the maps in KK-theory induced by the inclusion of the coefficient algebra into the Pimsner algebra $i:C(\WP^n_q) \hookrightarrow O_{E^{\otimes d}} \simeq C(L_q(d))$.

We will refer to these two sequences as the \emph{Gysin sequences} (in $KK$-theory) for the \Cs algebraic quantum lens space $C(L_q^{2n+1}(d\cdot N_{\m};\m))$.

\subsection{Computing the K-theory and K-homology of quantum lens spaces}

We finish by describing how the exact sequences \eqref{eq:6tesKhom} and \eqref{eq:6tesKth} can be used to obtain information about the KK-theory groups of quantum weighted lens spaces. 

Even though those sequences exist for every choice of weight $\m$, we will now restrict our attention to the case of weight vectors satisfying the assumptions of Proposition \ref{prop:Kthweigh}. This will allow us to use the computations of the K-theory groups of the weighted projective spaces in our computations.

We will now state an easy corollary of the results contained in \cite{BS16}.
\begin{prop}
Let $\m$ be a weight vector satisfying the assumptions of Proposition \ref{prop:Kthweigh}. Then the \Cs algebra  $C(\WP^n(\m))$ is KK-equivalent to $\C^{1+m_1 +\dots + m_n}$.
\end{prop}
\begin{proof}
As a first step we use the fact that the UCT class is closed under extensions and contains the algebra of compact operators. Whenever the weight vector $\m$ satisfies the assumptions of Proposition \ref{prop:Kthweigh}, by iterated use of the exact sequence \eqref{eq:extW} we obtain that the \Cs algebraic weighted projective space is also in the UCT class. By Proposition \ref{prop:Kthweigh} its K-theory groups are isomoprhic to those of $\C^{1+m_1 +\dots + m_n}$. The claim follows from the fact that in the UCT class two \Cs algebras are KK-equivalent if and only if they have isomorphic K-theory groups (cf. \cite[Corollary~23.10.2]{Bl98}).
\end{proof}

Note that for $n=1$ the extension \eqref{eq:extW} admits a completely positive splitting, hence KK-equivalence follows from the fact that the algebra $C(\WP^1_q(\m))$ is isomorphic to the unital \Cs algebra ${\mathcal{K}^{m_1}} \oplus \C$. Explicit representatives for the two KK-equivalences were constructed in \cite{AKL16}.

For ease of notation we will denote by $M:=m_1+\dots+m_n$. We let $[I] \in KK(\C^{M+1}, C(\WP^n_q(\m)))$ and $[\Pi] \in KK(C(\WP^n_q(\m)),\C^M) $ be the two classes that implement the KK-equivalence between $\C^{M+1}$ and $C(\WP^n_q(\m))$, i.e. satisfy
\begin{equation}
\label{eq:KKequiv}
[I] \otimes_{C(\WP^n_q(\m))} [\Pi] = 1_{KK(\C^{M+1}, \C^{M+1})}, \quad  [\Pi] \otimes_{\C^{M+1}} [I] = 1_{KK(C(\WP^n_q(\m)), C(\WP^n_q(\m)))}.
\end{equation}
These KK-equivalences can be used to simplify the exact sequences \eqref{eq:6tesKth} and \eqref{eq:6tesKhom}. Indeed, one can use them to replace the KK-groups of $C(\WP^n_q(\m))$ with those of the vector space $\C^{M+1}$ and then use the natural isomorphisms \[ KK_i(C, \C^M) \simeq \bigoplus_{k=0}^{M+1} K^i(C) \quad \mbox{and} \quad  KK_i(\C^M,C) \simeq \bigoplus_{k=0}^{M+1} K_i(C), \quad i=0,1.\]

Tensoring the class of the Hilbert \Cs module $E$ with the KK-equivalences $[I]$ and $[\Pi]$ one gets a class
\begin{equation}
\label{eq:matrix}
[I] \otimes_{C(\WP^n_q(\m))} [E] \otimes_{C(\WP^n_q(\m))}[\Pi] \in KK(\C^{M+1},\C^{M+1}). 
\end{equation}

Since the ring $KK(\C^{M+1}, \C^{M+1}) \simeq \Mat_{M+1} (\Z)$, we can look at the corresponding matrix implementing the map \eqref{eq:matrix}, that we denote by $\mathrm{A}$. The six term exact sequence in \eqref{eq:6tesKth} becomes
\[
\xymatrix {\oplus_{i =0}^{M+1} K^0(C) \ar[r]^-{1 - \mathrm{A}^d} & \oplus_{i = 0}^{M+1} K^0(C) \ar[r] & KK_0\big(C,C(L_q(d))\big) \ar[d]\\
KK_1(C,C(L_q(d)))\ar[u] &\oplus_{i = 0}^{M+1} K^1(C) \ar[l] & \oplus_{i = 0}^{M+1} K^1(C) \ar[l]^-{1 - \mathrm{A}^d} }\, ,\]
while, denoting with $\mathrm{A}^t$ the transpose of $\mathrm{A}$, the six term exact sequence in \eqref{eq:6tesKhom} becomes
\[\xymatrix{
\oplus_{i = 0}^{M+1} K_0(C) \ar[d] & \oplus_{r = 0}^{M+1} K_0(C) \ar[l]^{1 - (\mathrm{A}^t)^d} &  KK_0\big(C(L_q(d)),C\big) \ar[l] \\
KK_1\big(C(L_q(d)),C\big) \ar[r] & \oplus_{i = 0}^{M+1} K_1(C) \ar[r]^-{1 - (\mathrm{A}^t)^d}& \oplus_{i = 0}^{M+1} K_1(C) \ar[u]}
 \, .\]
 
 For $C=\C$, using the fact that $K_1(\WP^n_q(\m))=K^1(\WP^n_q(\m)=0$, the corresponding exact sequences in K-theory and K-homology become of the form
 \[\xymatrix {0 \ar[r]& K_0\big(C(L_q(d)\big) \ar[r] & \Z^{M+1} \ar[r]^-{1 - \mathrm{A}^d} & \Z^{M+1} \ar[r] & K_0\big(C(L_q(d)\big) \ar[r] & 0}\, ,\]\
 and
 \[\xymatrix{
0& K^0\big(C(L_q(d))\big)\ar[l] & \Z^{M+1} \ar[l] & \Z^{M+1} \ar[l]^{1 - (\mathrm{A}^t)^d} &  K^0\big(C(L_q(d))\big) \ar[l] &0\ar[l]}
 \, .\]
 thus allowing for the computation of the K-theory and K-homology groups of the quantum lens spaces as kernels and cokernels of a suitable integer matrix.

\begin{thm}
Let $\m$ be a weight vector satisfying the assumptions of Proposition \ref{prop:Kthweigh}. Then for any $d \in \N$ we have that
\[
K_0\big( C(L_q(d)) \big) \simeq \mathrm{Coker}(1 - \mathrm{A}^d), \qquad K_1\big(C(L_q(d))\big) \simeq \mathrm{Ker}(1 - \mathrm{A}^d) \]
and
\[
K^0\big( C(L_q(dlk;k,l)) \big) \simeq \mathrm{Ker}(1 - (\mathrm{A}^t)^d ), \qquad K^1\big( C(L_q(d)) \big) \simeq \mathrm{Coker}(1 - (\mathrm{A}^t)^d) \, .
\]
\end{thm} 
It remains an open problem to describe the precise relationship of our matrix $\mathrm{A}$ with the matrix used in \cite{BS16} to compute the K-theory of quantum lens spaces. 
 
\section{Final remarks}
An interesting class of lens spaces that lies in the intersection of those studied in \cite{DAL15} and \cite{BS16} is that for which the weight vector $\m$ satisfies $m_0=1$ and $m_i=m$ for all $i=1, \dots, n$. The associated coprime weight vector $\p$ for which $\m=\p^{\sharp}$ is then $\p=(m,1,\dots,1)$. It is straightforward to check that the number of independent Fredholm modules constructed in \cite{DAL15}, given by the formula 
\[
1 + \sum_{k=1}^{n}p_0p_1\dots p_{k-1} , 
\]
equals in that case $1 +M$, the dimension of the K-homology group $K^0(\WP^n_q(\m))$. 

We are also able to give, at least for this special class of examples, a positive answer to the question left open at the end of \cite[Section~9]{DAL15}, where the authors asked whether the Fredholm module they constructed actually built a complete set of generators for the K-homology group $K^0(\WP^n_q(\m))$.

Moreover these Fredholm modules can be used to give an explicit expression for the KK-equivalences $[\Pi]$ in \eqref{eq:KKequiv}, using a construction similar to the one of \cite[Section~7.4]{AKL16}, thus allowing one to write the matrix $\mathrm{A}$ in the form of a matrix of pairings.

Those computations go beyond the scope of the present paper and we postpone them to a later, more detailed work, where we plan to also address the problem of finding explicit representatives of K-theory and K-homology classes under less restrictive conditions on the set of weights.

%
%



%

\end{document}